\documentclass[11pt]{article}

\title{\vspace{-40pt}Simple closed curves, finite covers of surfaces, and power
subgroups of $\Out(F_n)$}
\author{Justin Malestein\thanks{University of Oklahoma, Department of Mathematics, 601 Elm Ave, Norman, OK, 73019, \url{jmalestein@math.ou.edu}}
\and Andrew Putman\thanks{Department of Mathematics, University of Notre Dame, 279 Hurley Hall, Notre Dame, IN 46556, 
 \url{andyp@nd.edu}}}

\date{}

\usepackage{amsmath,amssymb,amsthm,amscd,amsfonts}
\usepackage{mathtools}
\usepackage{epsfig,paralist}
\usepackage[vmargin=0.95in, hmargin=1.25in]{geometry}
\usepackage[font=small,format=plain,labelfont=bf,up,textfont=it,up]{caption}
\usepackage{type1cm}
\usepackage{calc}
\usepackage{enumerate}
\usepackage{bm}
\usepackage[all,cmtip]{xy}
\usepackage{eucal}
\usepackage{paralist}
\usepackage{lmodern}
\usepackage[T1]{fontenc}
\usepackage{etoolbox}

\usepackage[bookmarks, bookmarksdepth=2, colorlinks=true, linkcolor=blue, citecolor=blue, urlcolor=blue]{hyperref}

\apptocmd{\thebibliography}{\raggedright}{}{}

\numberwithin{equation}{section}

\theoremstyle{plain}
\newtheorem{theorem}{Theorem}[section]
\newtheorem{maintheorem}{Theorem}
\newtheorem{proposition}[theorem]{Proposition}
\newtheorem{lemma}[theorem]{Lemma}

\newtheorem*{claim}{Claim}

\newtheorem{step}{Step}

\theoremstyle{definition}

\newtheorem{defn}[theorem]{Definition}
\newenvironment{definition}[1][]{\begin{defn}[#1]\pushQED{\qed}}{\popQED \end{defn}}

\newtheorem{rmk}[theorem]{Remark}
\newenvironment{remark}[1][]{\begin{rmk}[#1] \pushQED{\qed}}{\popQED \end{rmk}}

\newtheorem{eg}[theorem]{Example}
\newenvironment{example}[1][]{\begin{eg}[#1] \pushQED{\qed}}{\popQED \end{eg}}

\DeclareMathOperator{\Hom}{Hom}

\DeclareMathOperator{\Ker}{ker}

\DeclareMathOperator{\Image}{Im}

\DeclareMathOperator{\Mod}{Mod}

\DeclareMathOperator{\GL}{GL}

\newcommand\C{\ensuremath{\mathbb{C}}}
\newcommand\Z{\ensuremath{\mathbb{Z}}}
\newcommand\Q{\ensuremath{\mathbb{Q}}}

\newcommand\Field{\ensuremath{\mathbb{F}}}

\DeclareMathOperator{\HH}{H}
\DeclareMathOperator{\CC}{C}

\DeclareMathOperator{\Aut}{Aut}
\DeclareMathOperator{\Out}{Out}
\DeclareMathOperator{\Inn}{Inn}
\DeclareMathOperator{\Ind}{Ind}

\newcommand\Set[2]{\ensuremath{\{\text{#1 $|$ #2}\}}}

\DeclareMathOperator{\Ad}{Ad}
\DeclareMathOperator{\Th}{th}

\newcommand\Lie{\ensuremath{\mathcal{L}}}
\newcommand\FLie{\ensuremath{\mathcal{FL}}}
\newcommand\Ass{\ensuremath{\mathcal{A}}}

\newcommand{\modulo}[1]{\ (\mathrm{mod}\ #1)}

\newcommand\HSCC{\ensuremath{\HH^{\text{scc}}}}
\newcommand\HPrimp{\ensuremath{\HH^{\fO_p}}}
\newcommand\HO{\ensuremath{\HH^{\cO}}}
\newcommand\HOPrime{\ensuremath{\HH^{\cO'}}}

\newcommand\bd{\ensuremath{\mathbf{d}}}

\newcommand\bt{\ensuremath{\mathbf{t}}}
\newcommand\ba{\ensuremath{\mathbf{a}}}
\newcommand\be{\ensuremath{\mathbf{e}}}
\newcommand\bk{\ensuremath{\mathbf{k}}}

\newcommand\cB{\ensuremath{\mathcal{B}}}
\newcommand\cE{\ensuremath{\mathcal{E}}}
\newcommand\cK{\ensuremath{\mathcal{K}}}
\newcommand\cM{\ensuremath{\mathcal{M}}}
\newcommand\cO{\ensuremath{\mathcal{O}}}

\newcommand\tSigma{\ensuremath{\widetilde{\Sigma}}}

\newcommand\tgamma{\ensuremath{\widetilde{\gamma}}}

\newcommand\tX{\ensuremath{\widetilde{X}}}
\newcommand\te{\ensuremath{\widetilde{e}}}
\newcommand\tf{\ensuremath{\widetilde{f}}}
\newcommand\tast{\ensuremath{\widetilde{\ast}}}

\renewcommand\oe{\ensuremath{\overline{e}}}
\newcommand\ox{\ensuremath{\overline{x}}}
\newcommand\oy{\ensuremath{\overline{y}}}

\newcommand\oGamma{\ensuremath{\overline{\Gamma}}}

\newcommand\oPhi{\ensuremath{\overline{\Phi}}}

\newcommand\hGamma{\ensuremath{\widehat{\Gamma}}}
\newcommand\hsM{\ensuremath{\widehat{\mathcal{M}}}}

\newcommand\fG{\ensuremath{\mathfrak{G}}}
\newcommand\fB{\ensuremath{\mathfrak{B}}}
\newcommand\fS{\ensuremath{\mathfrak{S}}}
\newcommand\fO{\ensuremath{\mathfrak O}}

\newcommand\tfG{\ensuremath{\widetilde{\mathfrak{G}}}}

\begin{document}

\vspace{-10pt}
\maketitle

\vspace{-24pt}
\begin{abstract}
\noindent
We construct examples of finite covers of punctured surfaces where the first rational
homology is not spanned by lifts of simple closed curves.  More generally, for any
set $\cO \subset F_n$ which is contained in the union of finitely many $\Aut(F_n)$-orbits, we construct
finite-index normal subgroups of $F_n$ whose first rational homology is not
spanned by powers of elements of $\cO$.
These examples answer questions of Farb--Hensel, Kent, Looijenga, and March\'{e}.
We also show that the quotient of $\Out(F_n)$ by the subgroup generated
by $k^{\Th}$ powers of transvections often contains infinite order elements, strengthening
a result of Bridson--Vogtmann saying that it is often infinite.
Finally, for any set $\cO \subset F_n$ which is contained in the union of finitely many $\Aut(F_n)$-orbits, 
we construct integral linear representations of free groups that have
infinite image and map all elements of $\cO$ to torsion elements.
\end{abstract}

\section{Introduction}
\label{section:introduction}

Let $\Sigma_{g,n}$ be a genus $g$ surface with $n$ punctures and let
$\Mod_{g,n}$ be its mapping class group.  An important classical
tool for studying $\Mod_{g,n}$ is its action on $\HH_1(\Sigma_{g,n})$.  Recently
there has been a lot of interest in studying the action of $\Mod_{g,n}$ 
on the homology of finite covers of $\Sigma_{g,n}$.  See, for instance,
\cite{FarbHenselNYJ, FarbHenselIsrael, GrunewaldLubotzky2, GrunewaldLubotzky1, Hadari, Hadari2, Hadari3, Koberda, Koberda2,
KoberdaRamanujan, Liu, LooijengaPrym, McMullen, PutmanWieland, Sun}.  These representations
encode a lot of subtle information; for instance,
work of Putman--Wieland \cite{PutmanWieland} relates them to the virtual first
Betti number of $\Mod_{g,n}$.

\paragraph{Simple closed curve homology.}
Let 
$\pi\colon \tSigma \rightarrow \Sigma_{g,n}$ be a finite cover.
The {\em simple closed curve homology} of $\tSigma$, denoted
$\HSCC_1(\tSigma)$, is the subspace of
$\HH_1(\tSigma)$ spanned by the set
\[\Set{$[\tgamma] \in \HH_1(\tSigma)$}{$\tgamma$ is a component of $\pi^{-1}(\gamma)$ for a simple closed curve $\gamma$ on $\Sigma_{g,n}$}.\]
Another more algebraic way
to describe $\HSCC_1(\tSigma)$ is as follows.  Let $R \subset \pi_1(\Sigma_{g,n})$
be the subgroup corresponding to $\tSigma$.  We have 
$\HH_1(\tSigma) = \HH_1(R)$, and $\HSCC_1(\tSigma)$ is the span of the
set
\[\Set{$[x^k] \in \HH_1(R)$}{$x \in \pi_1(\Sigma_{g,n})$ is a simple closed curve and $k \geq 1$ is such that $x^k \in R$}.\]
In a similar way, we can define $\HSCC_1(\tSigma;A)$ for any commutative ring $A$.

\paragraph{Equality.}
A fundamental question about $\HH_1(\tSigma;A)$ is whether
$\HSCC_1(\tSigma;A) = \HH_1(\tSigma;A)$.  This seems to have been first
asked in print by March\'{e} \cite{Marche}, though Kent has informed us that
it also arose in her work on the congruence subgroup property (see Remark \ref{remark:kent} below).  
It has since been asked by
Farb--Hensel \cite{FarbHenselNYJ} and Looijenga \cite{LooijengaOberwolfach}, who
said that it was a ``bit of a scandal'' that the answer is not known.  The first
serious progress on this was work of Koberda--Santharoubane
\cite{KoberdaRamanujan}, who used TQFT representations of $\Mod_{g,n}$
to construct examples where 
$\HSCC_1(\tSigma;\Z) \neq \HH_1(\tSigma;\Z)$.  However, they were unable to
rule out the possibility that $\HSCC_1(\tSigma;\Z)$ is finite-index
in $\HH_1(\tSigma;\Z)$.  In other words, they could not replace $\Z$ with $\Q$.  
Another partial result is a theorem of Farb--Hensel 
\cite[Proposition 8.4]{FarbHenselNYJ} which says that
$\HSCC_1(\tSigma;\Q) = \HH_1(\tSigma;\Q)$ when the deck group
is abelian.  They also constructed examples of finite covers 
$\tSigma_{1,2} \rightarrow \Sigma_{1,2}$ where
$\HSCC_1(\tSigma_{1,2};\Q) \neq \HH_1(\tSigma_{1,2};\Q)$; however, they indicated
that despite an extensive computer search they were unable to find such examples
in genus $2$.

\paragraph{Inequality.}
Our first main theorem gives examples of covers of punctured surfaces in all genera where
$\HSCC_1(\tSigma;\Q) \neq  \HH_1(\tSigma;\Q)$.

\begin{maintheorem}
\label{maintheorem:scchomology}
For all $g \geq 0$ and $n \geq 1$ such that $\pi_1(\Sigma_{g,n})$ is nonabelian, there is a finite regular cover $\tSigma \to \Sigma_{g,n}$ with $\HSCC_1(\tSigma;\Q) \neq \HH_1(\tSigma;\Q)$. 
\end{maintheorem}

\begin{remark}
The degrees of our covers are enormous, so it
is not surprising that they could not be found via a computer search.
\end{remark}

\begin{remark}
We do not know how to remove the condition $n \geq 1$ from Theorem \ref{maintheorem:scchomology}.  The
issue is that our proof of the key Proposition \ref{proposition:centerpower} below makes use of delicate
results about free restricted Lie algebras.  To prove the analogous result in the closed case, we would
have to generalize these results to certain $1$-relator restricted Lie algebras.
\end{remark}

\paragraph{Branched covers.}
We can define $\HSCC_1(\tSigma;\Q)$ for branched covers $\tSigma \rightarrow \Sigma_{g,n}$ exactly like for
unbranched covers.
Though we are unable to remove the condition $n \geq 1$ from
Theorem \ref{maintheorem:scchomology}, we do want to point out the following consequence for closed
genus $g$ surfaces $\Sigma_g$.

\begin{maintheorem}
\label{maintheorem:branched}
For all $g \geq 2$, there exists a finite branched cover $\tSigma \to \Sigma_g$ with 
$\HSCC_1(\tSigma;\Q) \neq \HH_1(\tSigma;\Q)$.
\end{maintheorem}

\noindent
This result can be derived from Theorem \ref{maintheorem:scchomology} as follows.  Let
$\tSigma' \rightarrow \Sigma_{g,1}$ be the unbranched cover provided by Theorem \ref{maintheorem:scchomology}.
The surface $\tSigma'$ is a finite-genus surface with punctures $p_1,\ldots,p_k$.
A small oriented loop $\gamma$ surrounding the puncture of $\Sigma_{g,1}$ lifts to a collection of
loops $\tgamma_1,\ldots,\tgamma_k$ such that $\tgamma_i$ surrounds $p_i$.
Fill in all the punctures of $\tSigma'$ to form a closed surface $\tSigma$.  We then have a branched
cover $\tSigma \rightarrow \Sigma_g$ whose branch points are precisely the $p_i$ such that the
map $\tgamma_i \rightarrow \gamma$ is a cover of degree greater than $1$.  Since the homology
classes of the $\tgamma_i$ lie in $\HSCC_1(\tSigma';\Q)$ and generate the kernel of
$\HH_1(\tSigma') \rightarrow \HH_1(\tSigma;\Q)$, we see that 
\[\HH_1(\tSigma';\Q) / \HSCC_1(\tSigma';\Q) \cong \HH_1(\tSigma;\Q) / \HSCC_1(\tSigma;\Q).\]
Since the left hand side is nonzero, so is the right hand side, as desired.

\paragraph{More general result.}
In fact, we prove a much more general result than Theorem \ref{maintheorem:scchomology}.  To state
it, let $F_n$ be the free group on $n$ generators and let $\cO \subset F_n$.  For
a finite-index $R < F_n$ and an abelian group $A$, define $\HO_1(R;A)$ to be the span in $\HH_1(R;A)$ of the set
\[\Set{$[x^k] \in \HH_1(R;A)$}{$x \in \cO$ and $k \geq 1$ is such that $x^k \in R$}.\]
We prove the following theorem.

\begin{maintheorem}
\label{maintheorem:ohomology}
Let $n \geq 2$ and let $\cO \subset F_n$ be contained in the union of finitely many $\Aut(F_n)$-orbits.  
Then there exists a finite-index $R \lhd F_n$ with $\HO_1(R;\Q) \neq \HH_1(R;\Q)$.
\end{maintheorem}

\noindent
We highlight two special cases of Theorem \ref{maintheorem:ohomology}.

\begin{example}
Let $\cO$ be the set of {\em primitive elements} of $F_n$, that is, elements that lie in a free basis.
The group $\Aut(F_n)$ acts transitively on $\cO$.  In this case, for finite-index $R \lhd F_n$
the group $\HO_1(R;A)$ was introduced by
Farb--Hensel \cite{FarbHenselNYJ}, who called it the {\em primitive homology} of $R$ and 
asked whether or not $\HO_1(R;\Q) = \HH_1(R;\Q)$ always holds.  They constructed counterexamples to this
for $n = 2$.  They also proved that there was indeed equality in some situations.  For 
instance they could prove equality when $F_n / R$ is abelian and in some cases when $F_n/R$ is $2$-step nilpotent.
Theorem \ref{maintheorem:ohomology} provides a negative answer to their question for all $n \geq 2$.
\end{example}

\begin{example}
\label{example:stoscc}
We next explain why Theorem \ref{maintheorem:ohomology} is a vast generalization of 
Theorem \ref{maintheorem:scchomology}.  Let $g \geq 0$ and $n \geq 1$ be such that 
$F=\pi_1(\Sigma_{g,n})$ is nonabelian.  Let $\cO \subset F$ be the set of simple closed curves.  For
a finite-index $R \lhd F$, let $\tSigma_R$ be the associated finite cover of $\Sigma_{g,n}$, so
$\HH_1(R;\Q) = \HH_1(\tSigma_R;\Q)$ and $\HO_1(R;\Q) = \HSCC_1(\tSigma_R;\Q)$.  The set $\cO$
is the union of finitely many mapping class group orbits; see \cite[Chapter 1.3]{FarbMargalitPrimer}.
Since $\Aut(F)$ is larger than the mapping class group, $\cO$ is contained in the union 
of finitely many $\Aut(F)$-orbits.  Applying Theorem \ref{maintheorem:ohomology}, we obtain 
a finite-index $R \lhd F$ satisfying $\HO_1(R;\Q) \neq \HH_1(R;\Q)$.
Theorem \ref{maintheorem:scchomology} follows.  This can be generalized
in a wide variety of ways.  For instance, for some fixed $m \geq 0$ we could take $\cO$ to be the set
of all elements of $F$ that can be represented by a curve with at most $m$ self-intersections.
\end{example}

\begin{remark}
\label{remark:kent}
Let $g \geq 0$ and $n \geq 1$ be such that
$F=\pi_1(\Sigma_{g,n})$ is nonabelian.
In \cite{KentMathOverflow}, Kent asked whether the conclusion of Theorem \ref{maintheorem:ohomology} holds
for $\cO$ the set of all curves that do not fill $\Sigma_{g,n}$, i.e.\ curves whose complement contains a 
non simply connected component.  Since these are not contained in finitely many $\Aut(F)$-orbits, this question
is not addressed by Theorem \ref{maintheorem:ohomology}.
\end{remark}

\paragraph{$\mathbf{p}$-primitive homology.}
We now discuss a variant of Theorem \ref{maintheorem:ohomology}.
Fix a prime $p$.  Say that $x \in F_n$ is {\em $p$-primitive} if it maps to a nonzero element of $\HH_1(F_n;\Field_p)$. We denote the set of $p$-primitive elements by $\fO_p$.  Observe 
that primitive elements of $F_n$ are $p$-primitive for all primes $p$.
There are infinitely many $\Aut(F_n)$-orbits of $p$-primitive element of $F_n$, so we cannot use
Theorem \ref{maintheorem:ohomology} to construct finite-index $R \lhd F_n$ satisfying $\HPrimp_1(R;\Q) \neq \HH_1(R;\Q)$.
However, the following still holds.

\begin{maintheorem}
\label{maintheorem:primphomology}
For all $n \geq 2$ and primes $p$, there is a finite-index $R \lhd F_n$ with
$\HPrimp_1(R;\Q) \neq \HH_1(R;\Q)$.  Moreover, we can choose $R$ such that $F_n/R$ is a $p$-group.
\end{maintheorem}

\paragraph{Transvections.}
Using TQFT representations, Funar \cite{FunarTQFT} and Masbaum \cite{Masbaum}
proved that for $g \geq 2$
and $n \geq 0$ and $k \geq 11$, the subgroup of $\Mod_{g,n}$ generated by $k^{\text{th}}$
powers of Dehn twists is an infinite-index subgroup of $\Mod_{g,n}$.  In fact, their methods
show that the quotient of $\Mod_{g,n}$ by the subgroup generated by $k^{\Th}$ powers of Dehn twists
contains infinite-order elements.  Using some of the methods of the proof of
Theorem \ref{maintheorem:primphomology}, we will prove an analogous theorem for
$\Out(F_n)$.  The analogue in $\Out(F_n)$ of a Dehn twist is a {\em transvection}, which
is defined as follows.  Let $S$ be a basis for $F_n$ and let $x,y \in S$ be distinct
elements.  The transvection $\tau_{S,x,y}$ is then the automorphism 
$\tau_{S,x,y}\colon F_n \rightarrow F_n$ defined via the formula
\[\tau_{S,x,y}(z) = \begin{cases}
yz & \text{if $z=x$}, \\
z & \text{if $z \neq x$}\end{cases} \quad \quad \quad (z \in S).\]
All transvections are conjugate to each other, and transvections with $S$ the standard
basis for $F_n$ appear in the usual generating sets for $\Out(F_n)$.  

For $n \geq 2$ and $k \geq 1$, let $\fG_{n,k}$ be the subgroup of $\Out(F_n)$ generated by $k^{\Th}$ powers
of transvections.  Bridson--Vogtmann \cite{BridsonVogtmann} showed that the quotient $\Out(F_n) / \fG_{n,k}$ contains a copy
of the free Burnside group $\fB_{n-1,k}$ obtained by quotienting $F_{n-1}$ by the subgroup generated by all
$k^{\Th}$ powers.  In particular, $\Out(F_n) / \fG_{n,k}$ is infinite whenever $\fB_{n-1,k}$ is.  We will
strengthen this by showing that $\Out(F_n) / \fG_{n,k}$ often contains infinite order elements.

\begin{maintheorem}
\label{maintheorem:transvections}
For all $n \geq 2$, there exists an infinite set of positive numbers $\cK_n$ such that
$\Out(F_n) / \fG_{n,k}$ contains infinite-order elements for all $k \in \cK_n$.
\end{maintheorem}

\begin{remark}
The reader might object that we should also include the transvections $\tau'_{S,x,y}$ defined
via the formula
\[\tau'_{S,x,y}(z) = \begin{cases}
zy & \text{if $z=x$}, \\
z & \text{if $z \neq x$}\end{cases} \quad \quad \quad (z \in S).\]
But this would be superfluous: letting $S' = (S - \{x\}) \cup \{x^{-1}\}$, 
we have $\tau'_{S,x,y} = \tau_{S',x^{-1},y}^{-1}$.
\end{remark}

\begin{remark}
Theorem \ref{maintheorem:transvections} implies an analogous result for $\Aut(F_n)$.
Let $\tfG_{n,k}$ be the subgroup of $\Aut(F_n)$ generated by $k^{\text{th}}$ powers
of transvections.  
The surjection $\Aut(F_n) \rightarrow \Out(F_n)$ restricts to a surjection from $\tfG_{n,k}$
to $\fG_{n,k}$.  Thus $\Aut(F_n) / \tfG_{n,k}$ contains infinite-order elements whenever
$\Out(F_n)/\fG_{n,k}$ does.
\end{remark}

\begin{remark}
The proof of Theorem \ref{maintheorem:transvections} shows that we can take $\cK_n$
to be the set
\[\cK_n = \Set{$k$}{there exists a prime power $p^e$ dividing $k$ such that
$p^e > p(p-1)(n-1)$}.\]
We conjecture that there is some uniform $m \geq 2$ such that we can take
$\cK_n = \Set{$k$}{$k \geq m$}$.  Precisely explaining what it would take to prove this
using the techniques of this paper would require delving into the details of our proof, but
roughly speaking one would have to construct for each $k \geq m$ a finite quotient $G = F_n / R$
as in Theorem \ref{maintheorem:primphomology} such that $x^k = 1$ for all $x \in G$ that are the
image of a primitive element of $F_n$.  In our current construction, the orders of such elements
are forced to grow with $n$. 
\end{remark}

\paragraph{Infinite quotients of free groups.}
We conclude with a pair of interesting consequences of Theorems
\ref{maintheorem:ohomology} and \ref{maintheorem:primphomology}.  The first
is as follows.

\begin{maintheorem}
\label{maintheorem:squotient}
Let $n \geq 2$ and let $\cO \subset F_n$ be contained in the union of finitely many $\Aut(F_n)$-orbits.  
Then there exists an integral
linear representation $\rho\colon F_n \rightarrow \GL_d(\Z)$ with infinite image such that every
element of $\rho(\cO)$ has finite order.
\end{maintheorem}

\noindent
For instance, as in Example \ref{example:stoscc} we can use Theorem \ref{maintheorem:squotient}
to construct for all $g \geq 0$ and $n \geq 1$ with $\pi_1(\Sigma_{g,n})$ nonabelian an integral
linear representation
$\rho\colon \pi_1(\Sigma_{g,n}) \rightarrow \GL_d(\Z)$ with infinite image such that
for all simple closed curves $x \in \pi_1(\Sigma_{g,n})$, the image $\rho(x)$ has finite order.
A slightly weaker version of this with the
representation landing in $\GL_d(\C)$ instead of $\GL_d(\Z)$ was originally
proved using TQFT representations
by Koberda--Santharoubane \cite[Theorem 1.1]{KoberdaRamanujan}, who attribute the
question of whether such representations exist to Kisin and McMullen.
Unlike us, Koberda--Santharoubane could also deal with closed surfaces.  

Our second
result is the following variant of Theorem \ref{maintheorem:squotient}.  

\begin{maintheorem}
\label{maintheorem:primpquotient}
For all $n \geq 2$ and primes $p$, there exists an integral linear representation
$\rho\colon F_n \rightarrow \GL_d(\Z)$ with infinite image such that every element of
$\rho(\fO_p)$ has finite order.
\end{maintheorem}

\begin{remark}
For $n=2$, this
was originally proved by Zelmanov; see \cite[p.\ 140]{ZelmanovSurvey}.
\end{remark}

\begin{remark}
In Theorems \ref{maintheorem:squotient} and \ref{maintheorem:primpquotient}, the images of our
representations are virtually free abelian.
\end{remark}

\paragraph{Outline.}
We prove Theorem \ref{maintheorem:primphomology} in \S \ref{section:primphomology}, Theorem \ref{maintheorem:ohomology}
in \S \ref{section:ohomology}, Theorem \ref{maintheorem:transvections} in \S \ref{section:transvections}, and
Theorems \ref{maintheorem:squotient} and \ref{maintheorem:primpquotient}
in \S \ref{section:quotients}.  As we indicated above, Theorem \ref{maintheorem:scchomology} follows
from Theorem \ref{maintheorem:ohomology} and Theorem \ref{maintheorem:branched} follows from
Theorem \ref{maintheorem:scchomology}, so this completes the proofs of all of our main theorems.

\paragraph{Acknowledgments.}
We wish to thank Khalid Bou-Rabee, Martin Bridson, Mikhail Ershov, Benson Farb, Daniel Groves, Sebastian Hensel, Dawid Kielak, Eduard Looijenga, and Karen Vogtmann for
helpful correspondence and conversations. Andrew Putman is supported in part by National Science Foundation grants DMS-1255350 and DMS-1737434.

\section{\texorpdfstring{$\mathbf{p}$}{p}-primitive homology: Theorem \ref{maintheorem:primphomology}}
\label{section:primphomology}

This section contains the proof of Theorem \ref{maintheorem:primphomology}.  It has four parts.
In \S \ref{section:certificate}, we give a criterion for certifying that $\HPrimp_1(R;\Q) \neq \HH_1(R;\Q)$ for a finite-index
subgroup $R \lhd F_n$.
In \S \ref{section:reductiongroup}, we reduce Theorem \ref{maintheorem:primphomology}
to Proposition \ref{proposition:centerpower}, which asserts that a finite $p$-group with certain special properties exists.
In \S \ref{section:lie}, we review some basic material on $p$-restricted Lie algebras.
Finally, in \S \ref{section:centerpower} we prove Proposition \ref{proposition:centerpower}.

\subsection{Certifying the insufficiency of \texorpdfstring{$\mathbf{p}$}{p}-primitive homology}
\label{section:certificate}

This section gives a criterion for certifying that $\HPrimp_1(R;\Q) \neq \HH_1(R;\Q)$ for finite-index
subgroups $R \lhd F_n$.  This criterion is a variant of one identified by Farb--Hensel \cite{FarbHenselNYJ}.
Our main result is as follows.

\begin{theorem}
\label{theorem:primpcriterion}
For some $n \geq 2$, consider a finite-index $R \lhd F_n$, a field $\bk$ of
characteristic $0$, and a prime $p$.  Letting $G = F_n/R$, assume that there exists a
$\bk$-representation $V$ of $G$ such that for all $p$-primitive $x \in F_n$, the action
on $V$ of the image of $x$ in $G$ fixes no nonzero vectors.  Then
$\HPrimp_1(R;\Q) \neq \HH_1(R;\Q)$.
\end{theorem}

The proof of Theorem \ref{theorem:primpcriterion} requires one preliminary result.  Consider
a normal subgroup $R \lhd F_n$ and a field $\bk$.  The conjugation action of $F_n$ on $R$ induces
an action of $F_n$ on $\HH_1(R;\bk)$.  The restriction of this action to $R$ is trivial, so
we obtain an induced action of $G = F_n/R$ on $\HH_1(R;\bk)$.
We then have the following theorem of Gasch\"{u}tz \cite{Gaschutz}. 

\begin{theorem}
\label{theorem:gaschutz}
For some $n \geq 1$, consider a finite-index $R \lhd F_n$ and a field $\bk$ of
characteristic $0$.  Letting $G = F_n/R$, the
$G$-module $\HH_1(R;\bk)$ is isomorphic to $\bk \oplus \left(\bk[G]\right)^{n-1}$.
\end{theorem}

\begin{proof}[Proof of Theorem \ref{theorem:primpcriterion}]
Passing to an irreducible subrepresentation of $V$, we can assume that $V$ is irreducible.
Since $\HPrimp_1(R;\bk) = \HPrimp_1(R;\Q) \otimes_{\Q} \bk$ and
$\HH_1(R;\bk) = \HH_1(R;\Q) \otimes_{\Q} \bk$, it is enough to prove that
$\HPrimp_1(R;\bk) \neq \HH_1(R;\bk)$.  Let $W \subset \HH_1(R;\bk)$ be the $V$-isotypic component.
Theorem \ref{theorem:gaschutz} implies that $W \neq 0$, so it is enough to prove that
the projection of $\HPrimp_1(R;\bk)$ to $W$ is $0$.
Consider a $p$-primitive element $x \in F_n$ and let $m \geq 1$ be
such that $x^m \in R$.  We must prove that the projection of
$[x^m] \in \HH_1(R;\bk)$ to $W$ is trivial.  Let $g \in G$ be the image of $x$.  The fact
that $x$ commutes with $x^m$ implies that $g$ acts trivially on $[x^m] \in \HH_1(R;\bk)$.  Since
$x$ is $p$-primitive, our assumptions imply that the only vector in $V$ that is fixed by $g$ is $0$.
We conclude that the projection of $[x^m]$ to $W$ is $0$, as desired.
\end{proof}

\subsection{Reduction: \texorpdfstring{$p$}{p}-groups with special centers}
\label{section:reductiongroup}

In this section, we reduce Theorem \ref{maintheorem:primphomology} to the following proposition,
which we will prove in \S \ref{section:centerpower} below.

\begin{proposition}
\label{proposition:centerpower}
For $n,p \geq 2$ with $p$ prime, there exists a finite $p$-group $G$, a central subgroup
$C$ of $G$, and a homomorphism $\Psi\colon C \rightarrow \Z/p$ such that the following hold.
\begin{compactitem}
\item $\HH_1(G;\Field_p) = \Field_p^n$.
\item For all $g \in G$ whose image in $\HH_1(G;\Field_p)$ is nontrivial, some power
of $g$ is in $C - \Ker(\Psi)$.
\end{compactitem}
\end{proposition}

\begin{remark}
To give some sense for what is going on in Proposition \ref{proposition:centerpower},
we give some easy examples of groups $G$ that satisfy its conclusion for small
values of $n$ and $p$.  In these examples, the central subgroup $C$ satisfies
$C \cong \Z/p$ and we can take $\Psi\colon C \rightarrow \Z/p$ to be the identity.
\begin{compactenum}
\item For any prime $p$, the cyclic group of order $p$ satisfies the conclusions of Proposition
\ref{proposition:centerpower} for $n=1$.  In this case, the subgroup $C$ is the entire group.
\item The $8$-element quaternion group satisfies the conclusions
of Proposition \ref{proposition:centerpower} for $n=2$ and $p=2$.  In this
case, the subgroup $C$ is the center, which is cyclic of order $2$.
\end{compactenum}
It is much harder to prove Proposition \ref{proposition:centerpower} for $n \geq 3$.
The issue is that in both of the above examples a stronger conclusion holds:
for {\em every} nontrivial $g \in G$ some power of $g$ lies in
$C - \Ker(\Psi)$.  One can show that there are no examples satisfying
this stronger condition for $n \geq 3$.  Indeed, given a group $G$ satisfying this
stronger condition one can use the construction in the proof
of Theorem \ref{maintheorem:primphomology} below to construct a $\C$-representation $V$
of $G$ such that no nontrivial element of $G$ fixes any nontrivial vector in $V$.  From this,
one can deduce that all abelian subgroups of $G$ are cyclic.  This implies that $n \leq 2$; see
\cite[Theorem 6.12]{Isaacs}.  See \cite{Mathoverflow} for more details.
\end{remark}

\begin{proof}[Proof of Theorem \ref{maintheorem:primphomology}, assuming Proposition \ref{proposition:centerpower}]
Let us first recall the setup.  Let $n \geq 2$ and let $p$ be a prime.  Our goal is to
construct a finite-index $R \lhd F_n$ such that $\HPrimp_1(R;\Q) \neq \HH_1(R;\Q)$ and
such that $F_n/R$ is a $p$-group.

Let $G$ and $C$ and $\Psi\colon C \rightarrow \Z/p$ 
be as in Proposition \ref{proposition:centerpower}.  Since 
$\HH_1(G;\Field_p) = \Field_p^n$, we can choose a homomorphism $\rho\colon F_n \rightarrow G$ 
taking a basis of $F_n$ to a set $S$ of elements of $G$ that projects to a basis
for $\HH_1(G;\Field_p)$.  Equivalently, the induced map 
$\rho_{\ast}\colon \HH_1(F_n;\Field_p) \rightarrow \HH_1(G;\Field_p)$
is an isomorphism.  Note that $\HH_1(G;\Field_p) = G/D$ with $D = G^p [G, G]$.
Since $S$ projects to a basis for $\HH_1(G;\Field_p)$, the group $G$ is generated
by $S \cup D$.
Since $G$ is a finite $p$-group, the group $D$ is the Frattini subgroup 
of $G$ (see \cite[Theorem 5.48]{RotmanBook}).  Since $S \cup D$ generates $G$ and
$D$ is the Frattini subgroup of $G$, we conclude that $S$ generates $G$, so $\rho$
is surjective.

If $x \in F_n$ is $p$-primitive, then $x$ projects to a nonzero element of
$\HH_1(F_n;\Field_p)$ and thus $\rho(x) \in G$ also projects to a nontrivial element
of $\HH_1(G;\Field_p)$.  Set $R = \ker(\rho)$.  By Theorem \ref{theorem:primpcriterion},
to prove that $\HPrimp_1(R;\Q) \neq \HH_1(R;\Q)$,
it is enough to construct a $\C$-representation $V$ of $G$ such
that for all $g \in G$ which project to a nonzero element of $\HH_1(G;\Field_p)$, the
action of $g$ on $V$ fixes no nonzero vectors.

By regarding $\Z/p$ as the set of $p^{\Th}$ roots of unity in $\C$, we can view
$\Psi\colon C \rightarrow \Z/p$ as a homomorphism from $C$ to $\C^{\ast}$.  
Let $W$ be the $1$-dimensional $\C$-representation of $C$ such that $c \in C$ acts
on $W$ as multiplication by $\Psi(c) \in \C^{\ast}$.  Define
\[V = \Ind_C^G W.\]
Consider $g \in G$ that projects to a nonzero element of $\HH_1(G;\Field_p)$.  We wish
to prove that $g$ fixes no nonzero vectors in $V$.  By assumption, some power of
$g$ equals an element $c \in C - \Ker(\Psi)$.  It is enough to prove that
$c$ fixes no nonzero vectors in $V$.  Choosing a set $\Lambda \subset G$
of coset representatives for $G/C$, we have
\[V = \bigoplus_{\lambda \in \Lambda} \lambda \cdot W.\]
Moreover, since $C$ is a central subgroup of $G$ it follows that for $w \in W$ 
and $\lambda \in \Lambda$ we have
\[c \cdot (\lambda \cdot w) = \lambda \cdot (c \cdot w) = \lambda \cdot (\Psi(c) w) = \Psi(c) (\lambda \cdot w).\]
We deduce that $c$ acts on $V$ as multiplication by $\Psi(c)$.  Since
$c \notin \Ker(\Psi)$, we conclude that
$c$ fixes no nonzero vectors in $V$, as desired.
\end{proof}

\subsection{Restricted Lie algebras}
\label{section:lie}

Before we prove Proposition \ref{proposition:centerpower}, we will need to discuss
some preliminary facts about free groups and Lie algebras that can be viewed as
``mod-$p$'' analogues of the familiar connection between the lower central series
of a free group and the free Lie algebra (see, e.g., \cite{SerreLie}).

The starting point is
the following definition, which was first made by Zassenhaus \cite{Zassenhaus}.

\begin{definition}
Let $\Gamma$ be a group and let $p$ be a prime.  The {\em Zassenhaus $p$-central series} of
$\Gamma$ is the fastest descending series
\[\Gamma = \gamma_1^p(\Gamma) \supset \gamma_2^p(\Gamma) \supset \gamma_3^p(\Gamma) \supset \cdots\]
satisfying the following two conditions:
\begin{compactitem}
\item $[\gamma_i^p(\Gamma),\gamma_j^p(\Gamma)] \subset \gamma_{i+j}^p(\Gamma)$ for all
$i,j \geq 1$.
\item For all $x \in \gamma_i^p(\Gamma)$, we have $x^p \in \gamma_{ip}^p(\Gamma)$.\qedhere
\end{compactitem}
\end{definition}

\begin{remark}
Explicitly, one can inductively define $\gamma_i^p(\Gamma)$ as that subgroup generated
by 
\begin{compactitem}
	\item $[x,y]$ for all $x \in \gamma_j^p(\Gamma), y \in \gamma_k^p(\Gamma)$ with
$j, k < i$ and $j + k \geq i$, and
	\item $x^p$ for all $x \in \gamma_j^p(\Gamma)$ with $j < i$ and $pj \geq i$. \qedhere
\end{compactitem}
\end{remark}

Given a group $\Gamma$, we define
\[\Lie_{i}^p(\Gamma) = \gamma_i^p(\Gamma) / \gamma_{i+1}^p(\Gamma) \quad \quad (i \geq 1).\]
The second condition in the definition of the Zassenhaus $p$-central series ensures
that $\Lie_{i}^p(\Gamma)$ is an $\Field_p$-vector space.  Define
\[\Lie^p(\Gamma) = \bigoplus_{i \geq 1} \Lie_i^p(\Gamma).\]
The commutator bracket on $\Gamma$ descends to an operation on $\Lie^p(\Gamma)$ that
endows it with the structure
of a graded Lie algebra over $\Field_p$.  More precisely, consider 
$x \in \gamma_i^p(\Gamma)$ and $y \in \gamma_j^p(\Gamma)$.  Letting
$\ox \in \Lie_i^p(\Gamma)$ and $\oy \in \Lie_j^p(\Gamma)$ be their images, the
Lie bracket $[\ox,\oy] \in \Lie_{i+j}^p(\Gamma)$ is the image of the commutator
bracket $[x,y] \in \gamma_{i+j}^p(\Gamma)$.

In fact, even more is true: Zassenhaus 
proved that $\Lie^p(\Gamma)$ is what is called a $p$-restricted Lie algebra, the definition of which
is as follows (\cite{Zassenhaus};
see \cite[\S 12]{DixonSautoyMannSegal} for a textbook reference).  We recommend that the reader not dwell on the three conditions in this
definition -- the only one we will explicitly use is the first.

\begin{definition} \label{definition:pliealg}
Fix a prime $p$.  A {\em $p$-restricted Lie algebra} over $\Field_p$ is a Lie algebra $A$ over
$\Field_p$ equipped with a $p^{\Th}$ power operation that takes
$x \in A$ to $x^{[p]} \in A$.  This operation must satisfy the following three conditions:
\begin{compactenum}
\item For $c \in \Field_p$ and $x \in A$, we have $\left(c x\right)^{[p]} = c^p x^{[p]}$.
\item For all $x \in A$, we have $\Ad(x^{[p]}) = \left(\Ad\left(x\right)\right)^p$,
where the right hand side indicates that we are taking the $p^{\text{th}}$
iterate of $\Ad(x)\colon A \rightarrow A$.
\item For $x, y \in A$, we have
\[(x+y)^{[p]} = x^{[p]} + y^{[p]} + \sum_{i=1}^{p-1} s_i(x,y),\]
where $s_i(x,y)$ is the coefficient of $t^{i-1}$ in the polynomial
$\left(\Ad\left(t x+y\right)\right)^{p-1}(x)$.\qedhere
\end{compactenum}
\end{definition}

\begin{remark}
If $A$ is a $p$-restricted Lie algebra, then for $x \in A$ and $k \geq 1$ we can define
$x^{[p^k]}$ by iterating the $p^{\Th}$ power operation $k$ times.
\end{remark}

The $p^{\Th}$-power operation on $\Lie^p(\Gamma)$ is induced by the operation of taking $p^{\Th}$ powers in $\Gamma$.  More precisely, consider $x \in \gamma_i^p(\Gamma)$.
Letting $\ox \in \Lie_i^p(\Gamma)$ be its image, the element
$\ox^{[p]} \in \Lie_{ip}^p(\Gamma)$ is the image of $x^p \in \gamma_{ip}^p(\Gamma)$.
In particular, the $p^{\Th}$ power operation on $\Lie^p(\Gamma)$ takes
$\Lie_i^p(\Gamma)$ to $\Lie_{ip}^p(\Gamma)$.

We now make the following definition.  

\begin{definition}
The {\em free $p$-restricted Lie algebra} on a set
$S$ is a $p$-restricted Lie algebra $\FLie^p(S)$ generated by $S$ such that for all
$p$-restricted Lie algebras $L$, we have
\[\Hom_{\text{Set}}(S,L) = \Hom_{\text{pLie}}(\FLie^p(S),L).\]
Here the right hand side is the set of morphisms of $p$-restricted Lie algebras.  
\end{definition}

\begin{remark}
See \cite[\S 2.7]{BahturinIdentical} for a textbook reference about $\FLie^p(S)$ 
that in particular proves that it exists.
\end{remark}

The free $p$-restricted Lie algebra $\FLie^p(S)$ has a natural grading that is respected
by the Lie bracket and the $p^{\Th}$ power operation.  Denoting the $i^{\Th}$ graded
piece by $\FLie_i^p(S)$, the degree $1$ piece $\FLie_1^p(S)$ is an $\Field_p$-vector
space with basis $S$.  The $p$-restricted Lie algebra $\FLie^p(S)$ 
is generated by $\FLie_1^p(S)$ in the sense that
its higher degree pieces are spanned by the result of repeatedly applying the
Lie bracket operation and the $p^{\Th}$ power operation to elements of $\FLie_1^p(S)$.

Lazard proved the following theorem connecting the free group and the free $p$-restricted
Lie algebra; see \cite[Theorem 6.5]{Lazard}.

\begin{theorem}
\label{theorem:lazardfree}
If $F$ is the free group on a set $S$ and $p$ is a prime, then $\Lie^p(F) \cong \FLie^p(S)$ as graded
$p$-restricted Lie algebras.
\end{theorem}

\subsection{The proof of Proposition \ref{proposition:centerpower}}
\label{section:centerpower}

In this section, we prove Proposition \ref{proposition:centerpower} and thus
complete the proof of Theorem \ref{maintheorem:primphomology}.  
We first recall the statement.  Consider $n,p \geq 2$ with $p$ prime.  We must
construct a finite $p$-group $G$, a central subgroup $C$ of $G$, and a homomorphism
$\Psi\colon C \rightarrow \Z/p$ such that the following two conditions hold.
\begin{compactitem}
\item $\HH_1(G;\Field_p) = \Field_p^n$.
\item For all $g \in G$ whose image in $\HH_1(G;\Field_p)$ is nontrivial, some power
of $g$ lies in $C - \Ker(\Psi)$.
\end{compactitem}
Let $S$ be an $n$-element set and let $F$ be the free group on $S$.  Pick $k \geq 1$
such that $p^k > (p-1)(n-1)$. (The reason for this assumption on $k$ will become clear later).
Set $G = F / \gamma_{p^k+1}^p(F)$, so
\[\HH_1(G;\Field_p) \cong G / \gamma_2^p(G) \cong F / \gamma_2^p(F) \cong \Field_p^n.\]
By Theorem \ref{theorem:lazardfree}, we have
\[\Lie^p(G) = \bigoplus_{i=1}^{p^k} \FLie^p_i(S).\]
Letting 
\[C = \gamma_{p^k}^p(G) \cong \gamma_{p^k}^p(F) / \gamma_{p^k+1}^p(F) \cong \FLie^p_{p^k}(S),\]
the fact that $\gamma_{p^k+1}^p(G) = 1$ implies that $C$ is central in $G$.  
The needed homomorphism $\Psi\colon C \rightarrow \Z/p$ is now provided by Proposition \ref{proposition:quotientlie}
below.

\begin{proposition}
\label{proposition:quotientlie}
Let $n \geq 2$, let $p$ be a prime, and let $S$ be an $n$-element set.
Pick $k \geq 1$ such that $p^k > (p-1)(n-1)$.
Then there exists an $\Field_p$-linear map
$\Psi\colon \FLie_{p^k}^p(S) \longrightarrow \Field_p$
such that $\Psi(v^{[p^k]}) \neq 0$ for all nonzero $v \in \FLie_1^p(S)$.
\end{proposition}
\begin{proof}
We begin with some preliminary observations.
Let $\Ass^p(S)$ be the free associative 
$\Field_p$-algebra on $S$.  The algebra $\Ass^p(S)$ can be viewed as consisting of
polynomials over $\Field_p$ in the noncommuting variables $S$, and thus has a natural
grading by degree.  Let $\Ass^p_i(S)$ be its $i^{\Th}$ graded piece, so
\[\Ass^p(S) = \bigoplus_{i=0}^{\infty} \Ass^p_i(S).\]
The associative algebra $\Ass^p(S)$ can be endowed with the
structure of a $p$-restricted Lie algebra via the bracket
\[[x,y] = xy - yx \quad \quad (x,y \in \Ass^p(S))\]
and the ordinary $p^{\Th}$ power operation
\[x^{[p]} = x^p \quad \quad (x \in \Ass^p(S)).\]
The inclusion $S \hookrightarrow \Ass^p(S)$ thus induces a homomorphism 
$\iota\colon \FLie^p(S) \rightarrow \Ass^p(S)$
of graded $p$-restricted Lie algebras.  
Though we will not need this, we remark that $\iota$ is injective; see 
\cite[Proposition 2.7.14]{BahturinIdentical}.

We have a commutative diagram
\[\xymatrix{
\FLie^p_1(S) \ar[r]^{\cong} \ar[d] & \Ass^p_1(S) \ar[d] \\
\FLie^p_{p^k}(S) \ar[r] & \Ass^p_{p^k}(S)}\]
whose horizontal arrows are $\iota$ and whose vertical arrows are the $p^{\Th}$-power operations.
To construct a linear map $\Psi\colon \FLie_{p^k}^p(S) \longrightarrow \Field_p$
such that $\Psi(v^{[p^k]}) \neq 0$ for all nonzero $v \in \FLie_1^p(S)$, it is thus enough to
construct a linear map
$\Phi\colon \Ass_{p^k}^p(S) \longrightarrow \Field_p$ such that
$\Phi(w^{p^k}) \neq 0$ for all nonzero $w \in \Ass_1^p(S)$.

Enumerate $S$ as $S = \{x_1,\ldots,x_n\}$.  For
a linear map $\Phi\colon \Ass_{p^k}^p(S) \longrightarrow \Field_p$, define
$f_{\Phi}\colon \Field_p^n \rightarrow \Field_p$ via the formula
\[f_{\Phi}(a_1,\ldots,a_n) = \Phi\left(\left(a_1 x_1 + \cdots + a_n x_n\right)^{p^k}\right) \quad \quad (a_1,\ldots,a_n \in \Field_p).\]
We must find some linear map $\Phi\colon \Ass_{p^k}^p(S) \longrightarrow \Field_p$
such that $f_{\Phi}(a_1,\ldots,a_n) \neq 0$ for all nonzero $(a_1,\ldots,a_n) \in \Field_p^n$. 

\begin{claim}
For any homogeneous polynomial $g \in \Field_p[t_1,\ldots,t_n]$ of degree $p^k$, there exists a linear
map $\Phi\colon \Ass_{p^k}^p(S) \longrightarrow \Field_p$ such that $f_{\Phi}(a_1,\ldots,a_n) = g(a_1,\ldots,a_n)$ for
all $(a_1,\ldots,a_n) \in \Field_p^n$.
\end{claim}
\begin{proof}[Proof of claim]
Set $\cE = \Set{$(e_1,\ldots,e_n) \in \Z_{\geq 0}^n$}{$e_1+\cdots+e_n = p^k$}$.  For $\be = (e_1,\ldots,e_n) \in \cE$,
define $\bt^{\be} = t_1^{e_1} \cdots t_n^{e_n} \in \Field_p[t_1,\ldots,t_n]$.  Also, for
$\ba = (a_1,\ldots,a_n) \in \Field_p^n$ define $\ba^{\be} = a_1^{e_1} \cdots a_n^{e_n} \in \Field_p$.
Write $g = \sum_{\be \in \cE} c_{\be} \bt^{\be}$ with $c_{\be} \in \Field_p$ for all $\be \in \cE$. 
The vector space $\Ass_{p^k}(S)$ has a basis
\[\cB = \Set{$x_{i_1}^{e_{i_1}} x_{i_2}^{e_{i_2}} \cdots x_{i_m}^{e_{i_m}}$}{$1 \leq i_1,\ldots,i_m \leq n$ and $e_{i_1}+\cdots+e_{i_m} = p^k$}\]
For $b \in \cB$, write $\bd(b) = (d_1(b),\ldots,d_n(b))$, where $d_i(b)$ is the number of $x_i$ factors that occur
in $b$.  For $\ba = (a_1,\ldots,a_n) \in \Field_p^n$, we have
\[(a_1 x_1 + \cdots + a_n x_n)^{p^k} = \sum_{b \in \cB} \ba^{\bd(b)} b.\]
Define $\Phi\colon \Ass_{p^k}^p(S) \longrightarrow \Field_p$ via the formula
\[\Phi(b) = \begin{cases}
c_{\be} & \text{if $b = x_1^{e_1} \cdots x_n^{e_n}$ for some $\be = (e_1,\ldots,e_n) \in \cE$},\\
0 & \text{otherwise}\end{cases} \quad \quad (b \in \cB).\]
For $\ba = (a_1,\ldots,a_n) \in \Field_p^n$, we then have
\begin{align*}
f_{\Phi}(a_1,\ldots,a_n) &= \Phi\left(\left(a_1 x_1 + \cdots + a_n x_n\right)^{p^k}\right) = \Phi\left(\sum_{b \in \cB} \ba^{\bd(b)} b\right) \\
&= \sum_{\be \in \cE} c_{\be} \ba^{\be} = g(a_1,\ldots,a_n). \qedhere
\end{align*}
\end{proof}

We must therefore construct a homogeneous polynomial $g \in \Field_p[t_1,\ldots,t_n]$ of degree $p^k$ such that
$g(a_1,\ldots,a_n) \neq 0$ for all nonzero $(a_1,\ldots,a_n) \in \Field_p^n$.  For this, it will be helpful to
be able to use a wider class of not necessarily homogeneous polynomials.
For $g,g' \in \Field_p[t_1,\ldots,t_n]$, write $g \sim g'$ if $g(a_1,\ldots,a_n) = g'(a_1,\ldots,a_n)$
for all $(a_1,\ldots,a_n) \in \Field_p^n$.  If $g = t_1^{e_1} \cdots t_n^{e_n}$ and $g' = t_1^{e_1'} \cdots t_n^{e_n'}$
for some $e_1,\ldots,e_n,e_1',\ldots,e_n' \geq 0$, we have $g \sim g'$ precisely when the following two conditions
hold for all $1 \leq i \leq n$:
\begin{compactitem}
\item $e_i \equiv e_i'\modulo{p-1}$, and
\item $e_i = 0$ if and only if $e_i' = 0$.
\end{compactitem}
We then have the following.  We remark that this claim is where we use our assumption that $p^k > (p-1)(n-1)$.

\begin{claim}
Consider a monomial $t_1^{e_1} \cdots t_n^{e_n}$ whose degree equals $1$ modulo $p-1$.  Then there exists a monomial
$t_1^{e_1'} \cdots t_n^{e_n'}$ of degree $p^k$ such that 
$t_1^{e_1} \cdots t_n^{e_n} \sim t_1^{e_1'} \cdots t_n^{e_n'}$.
\end{claim}
\begin{proof}[Proof of claim]
One of the $e_i$ must be nonzero.  Reordering the variables if necessary, we can assume that $e_1 \neq 0$.
Pick $e_1',\ldots,e_n'$ as follows.  For $2 \leq i \leq n$, let $e_i' = 0$ if $e_i=0$, and otherwise let $e_i'$ 
be the unique number satisfying $0<e_i'<p$ and $e_i' \equiv e_i \modulo{p-1}$.
Next, let $e_1' = p^k - (e_2' + \cdots + e_n')$.  Since $p^k > (p-1)(n-1) \geq e_2' + \cdots + e_n'$, the number 
$e_1'$ is positive.

It is clear that $e_1'+\cdots+e_n' = p^k$ and that $e_i' = 0$ if and only if $e_i = 0$ for all $1 \leq i \leq n$.
We must prove that $e_i \equiv e_i' \modulo{p-1}$ for all
$1 \leq i \leq n$.  The only nontrivial case is $i=1$.  For this, observe that modulo $p-1$ we have
\begin{align*}
e_1' &= p^k - (e_2' + \cdots + e_n') \equiv p^k - (e_2 + \cdots + e_n) = p^k + e_1 - (e_1+\cdots+e_n) \\
&\equiv p^k+e_1-1 \equiv e_1,
\end{align*}
where the next to last $\equiv$ follows from the fact that $e_1+\cdots+e_n \equiv 1 \modulo{p-1}$.
\end{proof}

By the above two claims, we see that the following claim implies the proposition. 

\begin{claim}
For all $n \geq 1$, there exists a polynomial $f_n \in \Field_p[t_1,\ldots,t_n]$ with the following properties.
\begin{compactitem}
\item The degree of each monomial appearing in $f_n$ equals $1$ modulo $p-1$.
\item $f_n(a_1,\ldots,a_n) \neq 0$ for all nonzero $(a_1,\ldots,a_n) \in \Field_p^n$.
\end{compactitem}
\end{claim}
\begin{proof}[Proof of claim]
For $f \in \Field_p[t_1,\ldots,t_n]$, define $Z(f) = \Set{$\ba \in \Field_p^n$}{$f(\ba)=0$}$.
Set $F(t_1,t_2) = t_1 - t_1 t_2^{p-1} + t_2$.  We claim that for
$g,h \in \Field_p[t_1,\ldots,t_n]$ we have
\begin{equation}
\label{eqn:zproperty}
Z(F(g,h)) = Z(g) \cap Z(h).
\end{equation}
It is clear that $Z(g) \cap Z(h) \subset Z(F(g,h))$, so we only need to prove
the other inclusion.  Consider $\ba \in Z(F(g,h))$.  If $\ba \notin Z(h)$, then
\[0 = F(g(\ba),h(\ba)) = g(\ba) - g(\ba) h(\ba)^{p-1} + h(\ba) = g(\ba) - g(\ba) + h(\ba) = h(\ba) \neq 0,\]
a contradiction.  We thus have $\ba \in Z(h)$, which implies that
\[0 = F(g(\ba),h(\ba)) = g(\ba) - g(\ba) h(\ba)^{p-1} + h(\ba) = g(\ba),\]
so $\ba \in Z(g)$ and thus $\ba \in Z(g) \cap Z(h)$, as desired.

We now construct $f_n$ by induction on $n$.  For the base case $n=1$, we simply set
$f_1 = t_1$.  Now assume that
$n>1$ and that $f_{n-1}$ has been constructed.  Define $f_n = F(f_{n-1},t_n)$.  The first
conclusion of the claim is clearly satisfied, and the second follows from
\eqref{eqn:zproperty}.
\end{proof}

This completes the proof of the proposition.
\end{proof}

\section{The proof of Theorem \ref{maintheorem:ohomology}}
\label{section:ohomology}

In this section, we prove Theorem \ref{maintheorem:ohomology}.
We first recall the setup.  Let $n \geq 2$ and let $\cO \subset F_n$ be contained
in the union of finitely many $\Aut(F_n)$-orbits.  Our goal is to construct
a finite-index $R \lhd F_n$ with $\HO_1(R;\Q) \neq \HH_1(R;\Q)$.
If $\cO' \subset F_n$ satisfies $\cO \subset \cO'$, then $\HO_1(R;\Q) \subset \HOPrime_1(R;\Q)$.  
From this, we see that without loss of generality we can enlarge $\cO$ and assume that $\cO$ is actually
equal to the union of finitely many $\Aut(F_n)$-orbits.  Let $\fS \subset F_n$ be the finite set
such that $\cO = \Aut(F_n) \cdot \fS$.  Without loss of generality, we can assume that $1 \notin \fS$.

The construction will have three steps.  
Recall that a subgroup of $F_n$ is characteristic if
it is preserved by all elements of $\Aut(F_n)$.

\begin{step}
For all $s \in \fS$, we construct a finite-index characteristic subgroup $R_1^s \lhd F_n$ with the following property.  Consider
an element $x \in F_n$ that is in the $\Aut(F_n)$-orbit of $s$.  Pick $m \geq 1$ such that $x^m \in R_1^s$.
Then $[x^m] \in \HH_1(R_1^s;\Q)$ is nonzero.
\end{step}

By Marshall Hall's Theorem (\cite{HallFactor}; see \cite{StallingsFinite} for a simple proof), there
exists a finite-index $T < F_n$ such that $s \in T$ and such that $s$ is a primitive element of $T$.  Define
\[R_1^s = \bigcap_{\phi \in \Aut(F_n)} \phi(T),\]
so $R_1^s$ is a finite-index subgroup of $F_n$ that is contained in $T$ and is characteristic.
Consider some $x \in F_n$ such that there
exists $\phi \in \Aut(F_n)$ with $\phi(x) = s$.  Pick $m \geq 1$ such that $x^m \in R_1^s$.  Our goal is
to prove that $[x^m] \in \HH_1(R_1^s;\Q)$ is nonzero.  Since $R_1^s$ is a characteristic subgroup of $F_n$, the
group $\Aut(F_n)$ acts on $R_1^s$ and thus on $\HH_1(R_1^s;\Q)$.  Observe that
\[\phi([x^m]) = [\phi(x)^m] = [s^m] \in \HH_1(R_1^s;\Q).\]
To prove that $[x^m] \in \HH_1(R_1^s;\Q)$ is nonzero, it is thus enough to prove that $[s^m] \in \HH_1(R_1^s;\Q)$
is nonzero.  The inclusion map $R_1^s \hookrightarrow T$ takes $s^m$ to $s^m$, and thus
the induced map $\HH_1(R_1^s;\Q) \rightarrow \HH_1(T;\Q)$ takes $[s^m] \in \HH_1(R_1^s;\Q)$ to $[s^m] \in \HH_1(T;\Q)$.
Since $s$ is a primitive element of $T$, it follows that $[s] \in \HH_1(T;\Q)$ is nonzero and hence that 
$[s^m] \in \HH_1(T;\Q)$ is nonzero.  We conclude that $[s^m] \in \HH_1(R_1^s;\Q)$ is nonzero, as desired.

\begin{step}
We construct a finite-index characteristic subgroup $R_1 \lhd F_n$ with the following property.  Consider $x \in \cO$.
Pick $m \geq 1$ such that $x^m \in R_1$.  Then, $[x^m] \in \HH_1(R_1;\Q)$ is nonzero.
\end{step}

Define
\[R_1 = \bigcap_{s \in \fS} R_1^s.\]
Since $R_1$ is a finite intersection of finite-index characteristic subgroups of $F_n$, it is also
a finite-index characteristic subgroup.
Consider $x \in \cO$ and some $m \geq 1$ such that $x^m \in R_1$.  We must prove that
$[x^m] \in \HH_1(R_1;\Q)$ is nonzero.  Let $s \in \fS$ be such that $x$ is in the $\Aut(F_n)$-orbit of $s$.  We
have $x^m \in R_1^s$, and by the previous step the element $[x^m] \in \HH_1(R_1^s;\Q)$ is nonzero.  An argument
like in the previous step now implies that $[x^m] \in \HH_1(R_1;\Q)$ is nonzero, as desired.

\begin{step}
We construct a finite-index $R \lhd F_n$ with $\HO_1(R;\Q) \neq \HH_1(R;\Q)$.
\end{step}

For each $s \in \fS$, pick $m_s \geq 1$ such that $s^{m_s} \in R_1$.  Since each $[s^{m_s}] \in \HH_1(R_1;\Q)$
is nonzero, there exists some large prime $p$ such that each $s^{m_s}$ projects to a nonzero element
of $\HH_1(R_1;\Field_p)$.  Applying Theorem \ref{maintheorem:primphomology} to $R_1$, we can find a finite-index
subgroup $R \lhd R_1$ such that $\HPrimp_1(R;\Q) \neq \HH_1(R;\Q)$, where here $\fO_p$ refers
to the $p$-primitive elements of $R_1$, {\it not} of $F_n$.
We claim that $\HO_1(R;\Q) \neq \HH_1(R;\Q)$, where by $\HO_1(R;\Q)$ we are considering $R$ as a subgroup of $F_n$.  
To prove this, it is enough to prove that $\HO_1(R;\Q) \subset \HPrimp_1(R;\Q)$.  Consider some $x \in \cO$.
Pick $\phi \in \Aut(F_n)$ and $s \in \fS$ such that $\phi(x) = s$.  Since $R_1$ is a
characteristic subgroup of $F_n$, the group $\Aut(F_n)$ acts on $R_1$ and thus on $\HH_1(R_1;\Field_p)$.  We
have
\[\phi([x^{m_s}]) = [\phi(x)^{m_s}] = [s^{m_s}] \in \HH_1(R_1;\Field_p).\]
Since $[s^{m_s}] \in \HH_1(R_1;\Field_p)$ is nonzero, so is $[x^{m_s}] \in \HH_1(R_1;\Field_p)$.  Pick $m \geq 1$
such that $(x^{m_s})^m \in R$.  We then have $[(x^{m_s})^m] \in \HPrimp_1(R;\Q)$, as desired.

\section{Transvections: Theorem \ref{maintheorem:transvections}}
\label{section:transvections}

In this section, we prove Theorem \ref{maintheorem:transvections}, which concerns the
subgroup of $\Out(F_n)$ generated by powers of transvections.  There are two sections.  In \S \ref{section:magnus},
we discuss the chain action, which will play a key role in our proof.  The proof itself is in
\S \ref{section:transvectionsproof}.

\subsection{The action of \texorpdfstring{$\Aut(F_n)$}{Aut(Fn)} on chains}
\label{section:magnus}

In this section, we construct a family of linear representations $\cM_R$
of subgroups of $\Aut(F_n)$ that are indexed by subgroups $R \lhd F_n$.
Our construction generalizes Suzuki's geometric construction
of the classical Magnus representation \cite{SuzukiGeometric}, which
corresponds to the case $R = [F_n,F_n]$.
Since we allow $F_n/R$ to be nonabelian, one can consider
$\cM_R$ as a kind of ``nonabelian Magnus representation''.

\paragraph{Setup.}
Fix a subgroup $R \lhd F_n$ and a field $\bk$.  Set $G = F_n/R$.  The conjugation action
of $F_n$ on $R$ induces an action of $G$ on $\HH_1(R;\bk)$.  
When $R$ is finite-index and $\bk$ has characteristic $0$,
this $G$-module is described by Theorem \ref{theorem:gaschutz} above.  Define
\[\Aut(F_n,R) = \Set{$\phi \in \Aut(F_n)$}{$\phi(R) = R$}.\]
The group $\Aut(F_n,R)$ acts on $\HH_1(R;\bk)$; however, this action can be very
complicated.  To help us understand it, we embed $\HH_1(R;\bk)$ into a larger
vector space.  This requires some topological preliminaries.

\paragraph{Graph homotopy-equivalences.}
Fix a free basis $S=\{x_1,\ldots,x_n\}$
for $F_n$.  Let $X_n$ be an oriented graph with a single vertex $\ast$
and with edges $\{e_1,\ldots,e_n\}$.  For $1 \leq i \leq n$,
the edge $e_i$ is an oriented loop based at $\ast$,
and we identify $F_n$ with $\pi_1(X_n,\ast)$ in such a way as to identify $x_i$ with
the homotopy class of the loop $e_i$.  The group $\Aut(F_n)$ can be identified
with the group of homotopy classes of homotopy equivalences of $X_n$ that
fix $\ast$.

\paragraph{Lifting homotopy-equivalences.}
Let $\pi\colon (\tX_n,\tast) \rightarrow (X_n,\ast)$ be
the based cover corresponding to $R \subset F_n$, so $\HH_1(\tX_n;\bk) = \HH_1(R;\bk)$.
For a continuous map $f\colon (X_n,\ast) \rightarrow (X_n,\ast)$, we know from covering space
theory that $f$ can be lifted to a map
$\tf\colon (\tX_n,\tast) \rightarrow (\tX_n,\tast)$ if and only if $f_{\ast}(R) \subset R$.
From this, we see that the group $\Aut(F_n,R)$ acts on $\tX_n$ by homotopy
equivalences that fix $\tast$.  The resulting action of
$\Aut(F_n,R)$ on $\HH_1(\tX_n;\bk) = \HH_1(R;\bk)$ is precisely the action arising
from the restriction of the action of $\Aut(F_n,R)$ to $R$.

\paragraph{Action on chains.}
Since $\tX_n$ is a $1$-dimensional cell complex, the
vector space $\HH_1(\tX_n;\bk)$ is a subspace of the cellular chain group
$\CC_1(\tX_n;\bk)$.  Namely,
\[\HH_1(\tX_n;\bk) = \ker(\CC_1(\tX_n;\bk) \stackrel{\partial}{\longrightarrow} \CC_0(\tX_n;\bk)).\]
The action of $\Aut(F_n,R)$ on $\tX_n$ by homotopy equivalences induces an
action of $\Aut(F_n,R)$ on $\CC_1(\tX_n;\bk)$ that restricts to the above
action on $\HH_1(\tX_n;\bk)$.  It turns out that $\CC_1(\tX_n;\bk)$ is far
easier to understand than $\HH_1(\tX_n;\bk)$. Note that this action on $1$-chains was
also studied in \cite{Hadari, Hadari2, Hadari3}.

\paragraph{The $G$-module structure.}
Observe that the action of the deck group $G = F_n/R$ on
$\tX_n$ endows each cellular chain group $\CC_k(\tX_n;\bk)$ with the structure
of a $G$-module.  It is easy to understand these $G$-modules:
\begin{compactitem}
\item We can identify the vertices of
$\tX_n$ with $G$ via the bijection taking $g \in G$ to $g(\tast)$.  Using
this identification, we obtain a $G$-equivariant isomorphism $\CC_0(\tX_n;\bk) \cong \bk[G]$.
\item For $1 \leq i \leq n$, let $\te_i$ be the oriented edge of $\tX_n$ that starts at $\tast$
and projects to $e_i \subset S_n$.  The edges of $\tX_n$ are precisely
$\Set{$g(\te_i)$}{$g \in G$, $1 \leq i \leq n$}$.  Using this, we obtain
a $G$-equivariant isomorphism $\CC_1(\tX_n;\bk) \cong (\bk[G])^n$.
\end{compactitem}
The boundary map $\partial\colon \CC_1(\tX_n;\bk) \rightarrow \CC_0(\tX_n;\bk)$
is $G$-equivariant, so $\HH_1(R;\bk) \subset \CC_1(\tX_n;\bk)$
is a $G$-submodule. This $G$-action agrees with the action of $G$ on $\HH_1(R;\bk)$ induced
by the conjugation action of $F_n$ on $R$ via the surjection $F_n \rightarrow G$.

\paragraph{The representation.}
Combining the above action of
$\Aut(F_n,R)$ on $\CC_1(\tX_n;\bk)$ with the above identification of
$\CC_1(\tX_n;\bk)$ with $(\bk[G])^n$, we obtain a homomorphism
\[\hsM_R\colon \Aut(F_n,R) \longrightarrow \Aut_{\bk}\left(\left(\bk[G]\right)^n\right).\]
Unfortunately, the image of $\hsM_R$
does {\em not} preserve the $G$-module structure on $(\bk[G])^n$.  Instead, we have
\[\hsM_R(\phi)(g \cdot v) = \phi_{\ast}(g) \cdot \hsM_R(\phi)(v) \quad \quad \left(\phi \in \Aut\left(F_n,R\right), g \in G, v \in \left(\bk[G]\right)^n\right),\]
where $\phi_{\ast} \in \Aut(G)$ is the induced action of $\phi \in \Aut(F_n,R)$ on
$G = F_n/R$.  To fix this, define
\[\Aut_R(F_n) = \Set{$\phi \in \Aut(F_n,R)$}{$\phi$ acts trivially on $G$}.\]
We then obtain a homomorphism
\[\cM_R\colon \Aut_R(F_n) \longrightarrow \Aut_{G}\left(\left(\bk[G]\right)^n\right).\]
Note that since the action on $1$-chains induces the action on homology, this representation $\cM_R$ is an extension of the
action on $\HH_1(R;\bk) \subset \CC_1(\tX_n;\bk) \cong (\bk[G])^n$.  For computations in
coordinates with $\cM_R$, we will need to use a basis.  The {\em standard basis} for
$(\bk[G])^n$ is the set $\{\oe_1,\ldots,\oe_n\}$, where $\oe_i \in (\bk[g])^n \cong \CC_1(\tX_n;\bk)$
is the chain corresponding to the oriented edge $\te_i$ of $\tX_n$.

\begin{remark}
The precise representation
$\Aut_R(F_n) \longrightarrow \Aut_{G}\left(\left(\bk[G]\right)^n\right)$ depends on the
choice of basis $S$ of $F_n$, but it is always an
extension of the action on $\HH_1(R;\bk)$.
\end{remark}

\paragraph{Image of a transvection.}
Recall that $S = \{x_1,\ldots,x_n\}$.  In the introduction we defined
the transvection $\tau_{S,x_1,x_2} \in \Aut(F_n)$ via the formula
\[\tau_{S,x_1,x_2}(x_i) = \begin{cases}
x_2 x_i & \text{if $i=1$},\\
x_i & \text{if $i \neq 1$}\end{cases} \quad \quad (1 \leq i \leq n).\]
Assume that $x_2^m \in R$ for some $m \geq 1$.
Examining our definitions, we see that $\tau_{S,x_1,x_2}^m \in \Aut_R(F_n)$.  The following
lemma calculates the image of $\tau_{S,x_1,x_2}$ under $\cM_R$.  

\begin{lemma}
\label{lemma:calctransvection}
For some $n \geq 1$,
let $R \lhd F_n$ be finite-index, let $\bk$ be a field, and let $\{x_1,\ldots,x_n\}$ be a basis for $F_n$.  Let
$G = F_n/R$, and for $1 \leq i \leq n$ let $g_i \in G$ be the image of $x_i \in F_n$.
Finally, let $\{\oe_1,\ldots,\oe_n\}$ be the standard basis for $(\bk[G])^n$.  If $m \geq 1$
is such that $x_1^m \in R$, then
\[\cM_R(\tau_{S,x_1,x_2}^m)(\oe_i) = \begin{cases}
\oe_i + \oe_2 + g_2\cdot \oe_2 + g_2^2 \cdot \oe_2 + \cdots + g_2^{m-1}\cdot \oe_2 & \text{if $i=1$},\\
\oe_i & \text{if $i \neq 1$}\end{cases} \quad \quad (1 \leq i \leq n).\]
\end{lemma}
\begin{proof}
The indicated calculation is only nontrivial for $i=1$.  For that case, observe that
$\tau_{S,x_1,x_2}^m(x_1) = x_2^m x_1$.  Using the concatenation product for paths,
the loop in $X_n$ corresponding to $x_2^m x_1$
is $e_2^m e_1$.  The lift of this to $\tX_n$ is the path
\[(\te_2) (g_2 \cdot \te_2) (g_2^2 \cdot \te_2) \cdots (g_2^{m-1} \cdot \te_2) (g_2^m \cdot \te_1) = (\te_2) (g_2 \cdot \te_2) (g_2^2 \cdot \te_2) \cdots (g_2^{m-1} \cdot \te_2) (\te_1).\]
The lemma follows.
\end{proof}

\subsection{Theorem \ref{maintheorem:transvections}}
\label{section:transvectionsproof}

We now give the proof of Theorem \ref{maintheorem:transvections}.
The heart of our
argument is the following more precise result for $\Aut(F_n)$.  For its statement, observe
that if $G$ is a finite group, $\pi\colon F_n \rightarrow G$ is a surjection with kernel $R$, and $V$ is an
irreducible representation of $G$ over a field $\bk$ of characteristic $0$, then the action
of $\Aut_R(F_n)$ on $\HH_1(R;\bk)$ preserves the $V$-isotypic component.

\begin{proposition}
\label{proposition:auttransvections}
For some $n \geq 2$, let $G$ be a finite group and let $\pi\colon F_n \rightarrow G$
be a surjection.  Assume that there exists an irreducible $\Q$-representation
$V$ of $G$ such that for all primitive $x \in F_n$, the action of $\pi(x)$ on $V$
fixes no nonzero vectors.  Let $R = \ker(\pi)$, let $W \subset \HH_1(R;\Q)$ be
the $V$-isotypic component, and let 
$\Phi\colon \Aut_{R}(F_n) \rightarrow \GL(W)$ be the restriction to $W$ of the action of
$\Aut_R(F_n)$ on $\HH_1(R;\Q)$.
Then the following hold.
\begin{compactitem}
\item The image of $\Phi$ has infinite order elements.
\item Let $m \geq 1$ be divisible by the orders in $G$ of all elements of $\Set{$\pi(x)$}{$x \in F_n$ primitive}$.
Then the $m^{\Th}$ power of any transvection
lies in $\Ker(\Phi)$.
\end{compactitem}
\end{proposition}
\begin{proof}
Farb--Hensel \cite[Theorem 1.1]{FarbHenselIsrael} proved that the $\Aut_R(F_n)$-orbit of a nonzero vector
in $\HH_1(R;\Q)$ is infinite. Moreover, they showed that for any nonzero $v \in \HH_1(R;\Q)$,
there is a $\phi \in \Aut_R(F_n)$ such that $\Set{$\phi^k(v)$}{$k \in \Z$}$ is infinite. Since the action of $\Aut_R(F_n)$ 
on $\HH_1(R;\Q)$ preserves $W$ (which
is nonzero by Theorem \ref{theorem:gaschutz}), we
deduce that the image of $\Phi$ has infinite order elements.

Now consider any basis $S = \{x_1,\ldots,x_n\}$ for $F_n$.  We must prove that
$\tau_{S,x_1,x_2}^m \in \Ker(\Phi)$.  Consider the representation
$\cM_R\colon \Aut_R(F_n) \rightarrow \Aut_G((\Q[G])^n)$ defined in \S \ref{section:magnus}.
Let $\{\oe_1,\ldots,\oe_n\}$ be the standard basis for $(\Q[G])^n$ and
for $1 \leq i \leq n$ define
$g_i = \pi(x_i) \in G$.  By Lemma \ref{lemma:calctransvection}, we have
\[\cM_R(\tau_{S,x_1,x_2}^m)(\oe_i) = \begin{cases}
\oe_i + \oe_2 + g_2\cdot \oe_2 + g_2^2\cdot \oe_2 + \cdots + g_2^{m-1}\cdot \oe_2 & \text{if $i=1$},\\
\oe_i & \text{if $i \neq 1$}\end{cases} \quad \quad (1 \leq i \leq n).\]
The element $\oe_2 + g_2\cdot \oe_2 + g_2^2\cdot \oe_2 + \cdots + g_2^{m-1}\cdot \oe_2$ of
$(\Q[G])^n$ is fixed by $g_2$, and thus projects to $0$ in the $V$-isotypic component
of $(\Q[G])^n$.  We conclude that $\cM_R(\tau_{S,x_1,x_2}^m)$ acts as the identity
on the $V$-isotypic component of $(\Q[G])^n$, which implies that 
$\tau_{S,x_1,x_2}^m \in \Ker(\Phi)$, as desired.
\end{proof}

In Proposition \ref{proposition:auttransvections}, we restrict ourselves to representations over $\Q$
since those are required in \cite[Theorem 1.1]{FarbHenselIsrael}.  The following lemma produces an 
appropriate representation over $\Q$ from one over an arbitrary field of characteristic $0$.

\begin{lemma}
\label{lemma:producerepresentation}
For some $n \geq 2$, let $G$ be a finite group and let $\pi\colon F_n \rightarrow G$
be a surjection.  Assume that for some field $\bk$ of characteristic $0$, 
there exists a $\bk$-representation
$V$ of $G$ such that for all primitive $x \in F_n$, the action of $\pi(x)$ on $V$
fixes no nonzero vectors.  Then there exists an irreducible $\Q$-representation $V'$ of
$G$ such that for all primitive $x \in F_n$, the action of $\pi(x)$ on $V'$
fixes no nonzero vectors.
\end{lemma}
\begin{proof}
Let $v \in V$ be a nonzero vector.  Since $\bk$ contains $\Q$ as a subfield, we can let $V''$
be the $\Q$-span in $V$ of the $G$-orbit of $v$.  The group $G$ acts on the $\Q$-vector space
$V''$, and any irreducible subrepresentation $V'$ of $V''$ satisfies the conclusion of the lemma.
\end{proof}

\begin{proof}[Proof of Theorem \ref{maintheorem:transvections}]
We first recall the statement.  Fix some $n \geq 2$, and for $k \geq 1$
let $\fG_{n,k} \subset \Out(F_n)$ be the subgroup generated by $k^{\text{th}}$ powers
of transvections.  We must prove that there exists an infinite set of positive
numbers $\cK_n$ such that $\Out(F_n) / \fG_{n,k}$ contains infinite-order elements for all $k \in \cK_n$.

Define
\[\cK_n = \Set{$k$}{there exists a prime power $p^e$ dividing $k$ such that $p^e > p(p-1)(n-1)$}.\]
To prove that $\cK_n$ satisfies the conclusion of the theorem, it is enough to prove that if $p^e$ is
a prime power satisfying $p^e > p(p-1)(n-1)$, then $\Out(F_n)/\fG_{n,p^e}$
contains infinite-order elements.  Lemma \ref{lemma:producerepresentation} 
and the proof of Theorem \ref{maintheorem:primphomology}
gives the following:
\begin{compactitem}
\item A surjection $\pi\colon F_n \rightarrow G$ to a finite group such that for all $p$-primitive
$x \in F_n$, the order of $\pi(x)$ is $p^e$.
\item An irreducible $\Q$-representation $V$ of $G$ such that for all $p$-primitive $x \in F_n$,
the action of $\pi(x)$ on $V$ fixes no nonzero vectors.
\end{compactitem}
We remark that these two bullet points hold in particular for each primitive $x \in F_n$.
Let $W$ be the $V$-isotypic component of $\HH_1(R;\Q)$ and let $\Phi\colon \Aut_R(F_n) \rightarrow \Aut_G(W)$
be the resulting action.  Define $\Gamma = \Image(\Phi)$.  Proposition \ref{proposition:auttransvections} implies that
$\Gamma$ contains infinite-order elements and that $(p^e)^{\Th}$ powers of transvections lie in $\Ker(\Phi)$.  
This implies that the quotient of $\Aut(F_n)$ by the subgroup generated by $(p^e)^{\Th}$ powers of transvections
contains infinite-order elements.

We will upgrade this to $\Out(F_n)$ as follows.  Define $\Out_R(F_n)$ to be the image of $\Aut_R(F_n)$ in
$\Out(F_n)$.  Unfortunately, it is not quite true that $\Phi$ factors through $\Out_R(F_n)$, but this almost holds.
Let $\Inn(F_n) \cong F_n$ be the group of
inner automorphisms and let $\Inn_R(F_n) = \Inn(F_n) \cap \Aut_R(F_n)$.  By definition, we have
\[\Inn_R(F_n) = \Set{$x \in F_n$}{$\pi(x) \in G$ is central}.\]
The image $\Phi(\Inn_R(F_n)) \subset \Gamma$ is a finite subgroup of the center of $\Gamma$.  Letting
\[\oGamma = \Gamma / \Phi(\Inn_R(F_n)),\] 
we obtain a 
surjective homomorphism 
$\oPhi\colon \Out_R(F_n) \rightarrow \oGamma$.
Since $\oGamma$ contains infinite-order elements and $\fG_{n,p^e} \subset \ker(\oPhi)$, the theorem follows.
\end{proof}

\section{Integral representations: the proofs of Theorems \ref{maintheorem:squotient} and \ref{maintheorem:primpquotient}}
\label{section:quotients}

We now prove Theorems \ref{maintheorem:squotient} and
\ref{maintheorem:primpquotient} via an argument of Koberda--Santharoubane 
\cite{KoberdaRamanujan}.  
We will give the details for Theorem \ref{maintheorem:primpquotient}; the proof of Theorem \ref{maintheorem:squotient}
is similar.  We start by recalling the statement
of Theorem \ref{maintheorem:primpquotient}.  Fix some $n \geq 2$ and some prime $p$.  We must construct
an integral linear representation $\rho\colon F_n \rightarrow \GL_d(\Z)$ with the following two properties.
\begin{compactitem}
\item The image of $\rho$ is infinite.
\item For all $p$-primitive $x \in F_n$, the image $\rho(x)$ has finite order.
\end{compactitem}
Use Theorem \ref{maintheorem:primphomology} to find a finite-index $R \lhd F_n$ such that
$\HPrimp_1(R;\Q) \neq \HH_1(R;\Q)$.  Define $G = F_n/R$ and $\hGamma = F_n/[R,R]$.  We thus have a short exact sequence
\[1 \longrightarrow \HH_1(R;\Z) \longrightarrow \hGamma \longrightarrow G \longrightarrow 1.\]
The action of $G$ on $\HH_1(R;\Z)$ preserves the subgroup $\HPrimp_1(R;\Z)$, and hence
$\HPrimp_1(R;\Z)$ is a normal subgroup of $\hGamma$.  Define $\Gamma = \hGamma / \HPrimp_1(R;\Z)$.  Setting
$A = \HH_1(R;\Z) / \HPrimp_1(R;\Z)$, we thus have a short exact sequence
\[1 \longrightarrow A \longrightarrow \Gamma \longrightarrow G \longrightarrow 1.\]
Since $\HPrimp_1(R;\Q)$ is a proper subspace of $\HH_1(R;\Q)$, the group $A$ is infinite.  Let
$\pi\colon F_n \rightarrow \Gamma$ be the projection, and consider some $p$-primitive $x \in F_n$.  Let $m$ be the order
of the image of $x$ in $G = F_n/R$.  It follows from the definition of
$\HPrimp_1(R;\Z)$ that $\pi(x^m) = 1$, and thus that $\pi(x)$ has finite order.

Since $A$ is a
finitely generated abelian group, there is a faithful
integral linear representation $A \hookrightarrow \GL_{d'}(\Z)$ for some $d' \geq 1$.  Since $A$ is a finite-index
subgroup of $\Gamma$, we can induce this representation up to $\Gamma$ to get a faithful integral linear
representation $\Gamma \hookrightarrow \GL_{d}(\Z)$ for some $d \geq 1$.  The desired integral linear representation
of $F_n$ is then the composition
\[F_n \stackrel{\pi}{\longrightarrow} \Gamma \hookrightarrow \GL_{d}(\Z). \qedhere\]

\begin{footnotesize}
\noindent
\begin{tabular*}{\linewidth}[t]{@{}p{\widthof{Department of Mathematics}+0.5in}@{}p{\linewidth - \widthof{Department of Mathematics} - 0.5in}@{}}
{\raggedright
Justin Malestein\\
University of Oklahoma\\
Department of Mathematics\\ 
601 Elm Ave\\
Norman, OK 73019\\
{\tt jmalestein@math.ou.edu}}
&
{\raggedright
Andrew Putman\\
Department of Mathematics\\
University of Notre Dame \\
279 Hurley Hall\\
Notre Dame, IN 46556\\
{\tt andyp@nd.edu}}
\end{tabular*}\hfill
\end{footnotesize}

\end{document}